\documentclass{article}

\input{paintbucket.sty}

\title{Paintbucket on graphs is PSPACE-complete}

\author{Ethan J. Saunders and Peter Selinger\\
  Dalhousie University}

\date{}

\begin{document}

\maketitle

\begin{abstract}
  The game of Paintbucket was recently introduced by Amundsen and
  Erickson. It is played on a rectangular grid of black and white
  pixels. The players alternately fill in one of their opponent's
  connected components with their own color, until the entire board is
  just a single color. The player who makes the last move wins. It is
  not currently known whether there is a simple winning strategy for
  Paintbucket. In this paper, we consider a natural generalization of
  Paintbucket that is played on an arbitrary simple graph, and we show
  that the problem of determining the winner in a given position of
  this generalized game is PSPACE-complete.
\end{abstract}

\section{Introduction}

The game of Paintbucket was recently introduced by Amundsen and
Erickson {\cite{AE2022}}. It is played on a rectangular grid of black
and white pixels. The players alternately fill in one of their
opponent's connected components with their own color, until the entire
board is just a single color. Here, two pixels are considered to be
connected if they share a common edge. The player who makes the final
move wins. An example game of Paintbucket is shown in
Figure~\ref{fig:paintbucket-example}(a). It is not currently known
whether there is a simple winning strategy for Paintbucket.

In this paper, we consider a natural generalization of Paintbucket
that is played on an arbitrary simple graph. Consider a simple
undirected graph $G$, given by a set $V$ of vertices and a set $E$ of
two-element subsets of $V$ called edges. We initially assign a color,
black or white, to every vertex. Two black vertices are connected if
there is a path (of length zero or greater) between them that only
passes through black vertices. A \emph{black group} is a non-empty
connected component of black vertices, i.e., a maximal connected set
of black vertices. White groups are defined analogously. A move by a
player consists of picking one of the opponent's groups and flipping
the colors of all of its vertices.  The winner is again the player who
makes the last move.  Note that Paintbucket played on a square grid
graph, as in Figure~\ref{fig:paintbucket-example}(b), is exactly the
same thing as the original version of Paintbucket played on a
rectangular grid of pixels.

\begin{figure}
  \[
  \begin{tikzpicture}[scale=0.3]
    \def\offset{8cm}
    \def\yoffset{-6.5cm}
    \begin{scope}
      \path (-2,1) node {(a)};
      \begin{scope}
        \fill[black] (-0.5,-0.5) rectangle +(3,3);
        \whitesquare(0,0);
        \whitesquare(0,2);
        \whitesquare(1,1);
        \whitesquare(2,0);
        \whitesquare(2,2);
        \draw[black] (-0.5,-0.5) rectangle +(3,3);
        \draw[->] (3.5,1) -- node[above]{Black} (6.5,1);
      \end{scope}
      \begin{scope}[xshift=\offset]
        \fill[black] (-0.5,-0.5) rectangle +(3,3);
        \whitesquare(0,0);
        \whitesquare(0,2);
        \whitesquare(1,1);
        \whitesquare(2,0);
        \draw[black] (-0.5,-0.5) rectangle +(3,3);
        \draw[->] (3.5,1) -- node[above]{White} (6.5,1);
      \end{scope}
      \begin{scope}[xshift={2*\offset}]
        \fill[black] (-0.5,-0.5) rectangle +(3,3);
        \whitesquare(0,0);
        \whitesquare(0,2);
        \whitesquare(1,1);
        \whitesquare(2,0);
        \whitesquare(1,0);
        \draw[black] (-0.5,-0.5) rectangle +(3,3);
        \draw[->] (3.5,1) -- node[above]{Black} (6.5,1);
      \end{scope}
      \begin{scope}[xshift={3*\offset}]
        \fill[black] (-0.5,-0.5) rectangle +(3,3);
        \whitesquare(0,2);
        \draw[black] (-0.5,-0.5) rectangle +(3,3);
        \draw[->] (3.5,1) -- node[above]{White} (6.5,1);
      \end{scope}
      \begin{scope}[xshift={4*\offset}]
        \fill[black] (-0.5,-0.5) rectangle +(3,3);
        \whitesquare(0,0);
        \whitesquare(1,0);
        \whitesquare(2,0);
        \whitesquare(0,1);
        \whitesquare(1,1);
        \whitesquare(2,1);
        \whitesquare(0,2);
        \whitesquare(1,2);
        \whitesquare(2,2);
        \draw[black] (-0.5,-0.5) rectangle +(3,3);
      \end{scope}
    \end{scope}
    \begin{scope}[yshift=\yoffset]
      \path (-2,1) node {(b)};
      \begin{scope}
        \foreach\i in {0,...,2} {
          \draw (0,\i) -- (2,\i+0);
          \draw (\i,0) -- (\i,2);
        }
        \blackvertex(0,0);
        \blackvertex(0,1);
        \blackvertex(0,2);
        \blackvertex(1,0);
        \blackvertex(1,1);
        \blackvertex(1,2);
        \blackvertex(2,0);
        \blackvertex(2,1);
        \blackvertex(2,2);
        \whitevertex(0,0);
        \whitevertex(0,2);
        \whitevertex(1,1);
        \whitevertex(2,0);
        \whitevertex(2,2);
        \draw[->] (3.5,1) -- node[above]{Black} (6.5,1);
      \end{scope}
      \begin{scope}[xshift=\offset]
        \foreach\i in {0,...,2} {
          \draw (0,\i) -- (2,\i+0);
          \draw (\i,0) -- (\i,2);
        }
        \blackvertex(0,0);
        \blackvertex(0,1);
        \blackvertex(0,2);
        \blackvertex(1,0);
        \blackvertex(1,1);
        \blackvertex(1,2);
        \blackvertex(2,0);
        \blackvertex(2,1);
        \blackvertex(2,2);
        \whitevertex(0,0);
        \whitevertex(0,2);
        \whitevertex(1,1);
        \whitevertex(2,0);
        \draw[->] (3.5,1) -- node[above]{White} (6.5,1);
      \end{scope}
      \begin{scope}[xshift={2*\offset}]
        \foreach\i in {0,...,2} {
          \draw (0,\i) -- (2,\i+0);
          \draw (\i,0) -- (\i,2);
        }
        \blackvertex(0,0);
        \blackvertex(0,1);
        \blackvertex(0,2);
        \blackvertex(1,0);
        \blackvertex(1,1);
        \blackvertex(1,2);
        \blackvertex(2,0);
        \blackvertex(2,1);
        \blackvertex(2,2);
        \whitevertex(0,0);
        \whitevertex(0,2);
        \whitevertex(1,1);
        \whitevertex(2,0);
        \whitevertex(1,0);
        \draw[->] (3.5,1) -- node[above]{Black} (6.5,1);
      \end{scope}
      \begin{scope}[xshift={3*\offset}]
        \foreach\i in {0,...,2} {
          \draw (0,\i) -- (2,\i+0);
          \draw (\i,0) -- (\i,2);
        }
        \blackvertex(0,0);
        \blackvertex(0,1);
        \blackvertex(0,2);
        \blackvertex(1,0);
        \blackvertex(1,1);
        \blackvertex(1,2);
        \blackvertex(2,0);
        \blackvertex(2,1);
        \blackvertex(2,2);
        \whitevertex(0,2);
        \draw[->] (3.5,1) -- node[above]{White} (6.5,1);
      \end{scope}
      \begin{scope}[xshift={4*\offset}]
        \foreach\i in {0,...,2} {
          \draw (0,\i) -- (2,\i+0);
          \draw (\i,0) -- (\i,2);
        }
        \blackvertex(0,0);
        \blackvertex(0,1);
        \blackvertex(0,2);
        \blackvertex(1,0);
        \blackvertex(1,1);
        \blackvertex(1,2);
        \blackvertex(2,0);
        \blackvertex(2,1);
        \blackvertex(2,2);
        \whitevertex(0,0);
        \whitevertex(1,0);
        \whitevertex(2,0);
        \whitevertex(0,1);
        \whitevertex(1,1);
        \whitevertex(2,1);
        \whitevertex(0,2);
        \whitevertex(1,2);
        \whitevertex(2,2);
      \end{scope}
    \end{scope}
    \begin{scope}[yshift=2*\yoffset]
      \path (-2,1) node {(c)};
      \begin{scope}
        \foreach\i in {0,...,2} {
          \draw (0,\i) -- (2,\i+0);
          \draw (\i,0) -- (\i,2);
        }
        \blackvertex(0,1);
        \blackvertex(1,0);
        \blackvertex(1,2);
        \blackvertex(2,1);
        \whitevertex(0,0);
        \whitevertex(0,2);
        \whitevertex(1,1);
        \whitevertex(2,0);
        \whitevertex(2,2);
        \draw[->] (3.5,1) -- node[above]{Black} (6.5,1);
      \end{scope}
      \begin{scope}[xshift=\offset]
        \edge(0,0)(0,2);
        \edge(0,0)(2,0);
        \edge(1,0)(1,1);
        \edge(0,1)(1,1);
        \edge(2,2)(0,2);
        \edge(2,2)(1,1);
        \edge(2,2)(2,0);
        \blackvertex(0,1);
        \blackvertex(1,0);
        \blackvertex(2,2);
        \whitevertex(0,0);
        \whitevertex(0,2);
        \whitevertex(1,1);
        \whitevertex(2,0);
        \draw[->] (3.5,1) -- node[above]{White} (6.5,1);
      \end{scope}
      \begin{scope}[xshift={2*\offset}]
        \edge(0,1)(0,2);
        \edge(0,2)(2,2);
        \edge(2,2)(1,0);
        \edge(1,0)(0,1);
        \blackvertex(0,1);
        \blackvertex(2,2);
        \whitevertex(1,0);
        \whitevertex(0,2);
        \draw[->] (3.5,1) -- node[above]{Black} (6.5,1);
      \end{scope}
      \begin{scope}[xshift={3*\offset}]
        \edge(0,2)(1,0);
        \blackvertex(1,0);
        \whitevertex(0,2);
        \draw[->] (3.5,1) -- node[above]{White} (6.5,1);
      \end{scope}
      \begin{scope}[xshift={4*\offset}]
        \whitevertex(1,0);
      \end{scope}
    \end{scope}
  \end{tikzpicture}
  \]
  \caption{(a) An example game of Paintbucket played by the authors. White
    makes the last move and therefore wins. (b) The same game, played
    on a graph. (c) The same game, played on a bipartite graph.}
  \label{fig:paintbucket-example}
\end{figure}
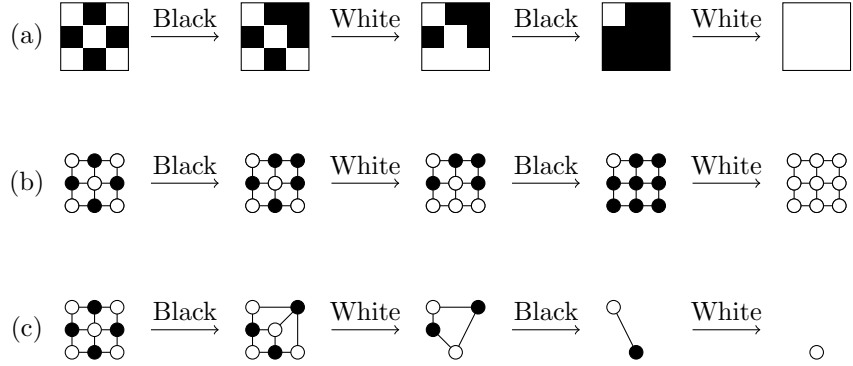

In this paper, we prove that the problem of determining the winner in
a given position of Paintbucket on graphs is PSPACE-complete.

\paragraph{Related work.}
Burke and Tennenhouse studied the game Flag Coloring \cite{BT2024},
which in the case of two colors can be thought of as an impartial
version of Paintbucket. They showed that Flag Coloring on graphs is
PSPACE-complete.

\section{Avoider-enforcer games}
\label{sec:ae-games}

We will show the PSPACE-completeness of Paintbucket on graphs by
reduction from a known PSPACE-complete problem, namely the decision
problem for \emph{avoider-enforcer games}, which we now define.

\begin{definition}
  The avoider-enforcer game is played on a pair $(C,\Aa)$, where $C$
  is a finite set of \emph{cells} and $\Aa=(A_j)_{j\in J}$ is a family
  of subsets of $C$, which we call the \emph{avoider sets}. The
  players, who are called the \emph{avoider} and the \emph{enforcer},
  take turns. On each turn, a player claims one cell, which afterwards
  cannot be claimed again (we can think of the players as ``coloring''
  the cells). The game finishes when all cells are colored, and the
  avoider loses if and only if they have colored all the elements of
  some avoider set.
\end{definition}

\begin{remark}\label{rem:alternate}
  The following is an equivalent description of the avoider-enforcer
  game. Given a position $(C,\Aa)$, a move by the avoider consists of
  choosing some $c\in C$, and then the new position is $(C',\Aa')$,
  where
  \[
  C' = C \setminus \s{c}
  \quad\mbox{and}\quad
  \Aa' = (A_j\setminus\s{c})_{j\in J}.
  \]
  A move by the enforcer in position $(C,\Aa)$ consists of choosing some
  $c\in C$, and then the new position is $(C'',\Aa'')$, where
  \[
  C'' = C \setminus \s{c}
  \quad\mbox{and}\quad
  \Aa'' = (A_j)_{j\in J,\,c\not\in A_j}.
  \]
  In other words, the avoider removes a cell from $C$ and from all
  avoider sets, whereas the enforcer removes a cell from $C$ and
  removes all avoider sets containing it. The game ends when
  $C=\emptyset$. In this case, either $\Aa=\emptyset$, in which case
  the avoider wins, or $\Aa=\s{\emptyset}$, in which case the enforcer
  wins. 
\end{remark}

The decision problem for avoider-enforcer games is the following:
given a position $(C,\Aa)$ and a player to move, decide whether the
avoider has a winning strategy. This problem is known to be
PSPACE-complete, even in the case where each avoider set consists of
exactly 6 cells {\cite{GO2023}}.

\begin{remark}\label{rem:even}
  The decision problem for avoider-enforcer games remains
  PSPACE-complete even if we restrict the possible instances to those
  where the number of cells $|C|$ is even and the avoider moves
  first. Indeed, given a position where the enforcer moves first, we
  can simply add one additional cell $c$ to $C$ and to none of the
  avoider sets, and let the avoider move first; it is easy to see that
  it is in the avoider's interest to play their first move at $c$,
  after which the game is equivalent to what we started with. We can
  therefore assume without loss of generality that the avoider moves
  first. Now, consider the case where $|C|$ is odd. We modify the game
  by adding one more cell $c'$ and exactly one additional avoider set
  $\s{c'}$. Clearly, the avoider loses if they ever play at $c'$, and
  it is also easy to see that the enforcer prefers every other move to
  $c'$. Therefore, we can assume without loss of generality that $c'$
  will be the last cell played; since it is the enforcer who plays
  last, this game is equivalent to the one that we started with.
\end{remark}

\section{Paintbucket on bipartite graphs}

The game of Paintbucket on graphs was defined in the introduction. We
now give an equivalent description where every connected component of
a color is contracted to a single vertex.

The game is played on a connected bipartite graph $(V_b, V_w, E)$,
where $V_b$ and $V_w$ are finite sets of vertices that we call
\emph{black} and \emph{white} vertices, respectively, and $E\seq
V_b\times V_w$ is a set of edges, each of which connects a black
vertex to a white one.

A move by Black consists of choosing one white vertex $v$, coloring it
black, and then immediately merging it with all of its (necessarily
black) neighbors. More precisely, by ``merging'', we mean that all of
the neighbors $w_1,\ldots,w_k$ of $v$ are deleted from the graph, and
new edges are added between $v$ and all the (necessarily white)
neighbors of $w_1,\ldots,w_k$. The resulting graph is again bipartite
and connected. A move by White is defined dually. The game ends when
the graph consists of a single vertex. Black wins if and only if the
final vertex is black.  Equivalently, the winner is the player who
made the last move.

It is easy to see that this description of Paintbucket using bipartite
graphs is equivalent to the game of Paintbucket on graphs described in
the introduction. For example, Figure~\ref{fig:paintbucket-example}(c)
shows the same game as Figure~\ref{fig:paintbucket-example}(a) and
(b), played in this setting. From now on, when we refer to
Paintbucket, we mean Paintbucket on bipartite graphs, unless otherwise
stated.

The following lemmas state some properties of Paintbucket that will
be useful to us later.

\begin{lemma}\label{lem:claw}
  Consider Paintbucket played on a bipartite graph $G$ with a
  distinguished white vertex $v$ to which $k$ black leaves are
  attached, as in Figure~\ref{fig:claw}.  Let $m$ be the number of
  white vertices of $G$. Then if $m\leq k$, Black has a first-player
  winning strategy. Moreover, if $m<k$, Black has a winning strategy
  regardless of who goes first.
\end{lemma}

\begin{figure}
  \[
  \begin{tikzpicture}[scale=0.3]
    \def\v{3}
    \def\w{1.5}
    \draw (0,0) -- (-2*\w,\v) node[above=1mm] {$1$};
    \draw (0,0) -- (-1*\w,\v) node[above=1mm] {$2$};
    \draw (0,0) -- (0*\w,\v) node[above=1mm] {$3$};
    \draw (0,0) -- (2*\w,\v) node[above=1mm] {$k$};
    \draw (0,-1) circle [x radius=2, y radius=1] node {$G$};
    \blackvertex(-2*\w,\v);
    \blackvertex(-1*\w,\v);
    \blackvertex(0*\w,\v);
    \path (1*\w,\v) node {$\cdots$};
    \blackvertex(2*\w,\v);
    \whitevertex(0,0);
  \end{tikzpicture}
  \]
  \caption{A bipartite graph with $k$ black leaves attached to a white
    vertex.}
  \label{fig:claw}
\end{figure}
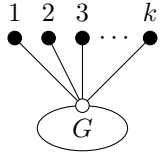

\begin{proof}
  Black's strategy is to play anywhere except $v$, unless $v$ is the
  only remaining white vertex. Therefore, each time Black plays, $m$
  decreases by exactly $1$. Each time White plays, $k$ decreases by at
  most $1$. Therefore, when $m=1$ and it is Black's turn, there is
  still at least one black leaf and Black wins by playing $v$.
\end{proof}

\begin{lemma}\label{lem:neighbors}
  In Paintbucket on a bipartite graph, if $n$ vertices have the same
  set of neighbors and a player plays at one of the $n$ vertices, the
  remaining ones become leaves that share the same neighbor.
\end{lemma}

\begin{proof}
  Suppose that $v_1,\ldots,v_n$ are black vertices that have the same
  neighbors $u_1,\ldots,u_k$, and suppose White plays at $v_1$. Then
  $v_1$ and $u_1,\ldots,u_k$ are all contracted into a single vertex
  $u$, which becomes the unique neighbor of each of $v_2,\ldots,v_n$
  in the resulting graph. So they all become leaves sharing the same
  neighbor $u$.
\end{proof}

\begin{lemma}\label{lem:complete}
  Consider Paintbucket played on the complete bipartite graph
  $K_{m,n}$, i.e., the bipartite graph with $m$ black vertices, $n$
  white vertices, and all possible edges. Then:
  \begin{enumerate}\alphalabels
  \item If $m,n>1$, the position is a second-player win.
  \item If $m>1$ and $n=1$, the position is winning for Black,
    no matter who goes first.
  \item If $m=1$ and $n>1$, the position is winning for White,
    no matter who goes first.
  \item If $m=n=1$, the position is a first-player win.
  \end{enumerate}
\end{lemma}

\begin{proof}
  (d) is obvious, and (b) and (c) follow by Lemma~\ref{lem:claw} and
  its dual. To prove (a), note that if Black plays first, Black
  collapses all black vertices and leaves $n-1$ white vertices
  remaining. White wins by playing at the only remaining black vertex.
  If White plays first, the situation is dual.
\end{proof}

\section{Paintbucket is PSPACE-complete}

\subsection{Statement of the main result}

The decision problem for Paintbucket on bipartite graphs is the
following: Given a connected bipartite graph and a player to move,
determine whether Black has a winning strategy. For concreteness, we
consider the problem size to be the number $n$ of vertices of the
graph. Note that it would be equally valid to use the number of edges,
because the number of edges in a connected graph is between $n-1$ and
$n^2$, which is polynomial in $n$. We will prove the following
theorem:

\begin{theorem}\label{thm:main}
  The decision problem for Paintbucket on bipartite graphs is
  PSPACE-complete.
\end{theorem}

\subsection{Translation from avoider-enforcer games to Paintbucket}

We prove the theorem by reducing the avoider-enforcer game to
Paintbucket. Given a position $(C,\Aa)$ of the avoider-enforcer game
and a natural number $K$, we will define a connected bipartite graph
$\G{K}{(C,\Aa)}$ as follows.

Let $I=|C|$ and enumerate the cells as $C=\s{c_1,\ldots,c_I}$. Also
let $\Aa=(A_j)_{j\in \s{1,\ldots,J}}$.  From now on, unless otherwise
specified, the indices $i$, $j$, and $k$ will always range over the
sets $\s{1,\ldots,I}$, $\s{1,\ldots,J}$, $\s{1,\ldots,K}$,
respectively. The graph $\G{K}{(C,\Aa)}$ will have vertices $v_i$,
$u_i$, $w_k$, $t_{j,k}$, $r$, and $s$. We assume all of these vertices
to be distinct. We refer to these vertices as being of \emph{type}
$v$, $u$, $w$, $t$, $r$, and $s$, respectively. We define
$\G{K}{(C,\Aa)}=(V_b,V_w,E)$, where
\[
\begin{array}{lll}
  V_b &=& \s{v_i} ~\cup~ \s{w_k} ~\cup~ \s{r}, \\
  V_w &=& \s{u_i} ~\cup~ \s{t_{j,k}} ~\cup~ \s{s}, \\
  E &=& \s{(r,u_i)} \\
  && {}\cup \s{(v_i,u_i)} \\
  && {}\cup \s{(v_i,s)}\\
  && {}\cup \s{(t_{j,k},r)} \\
  && {}\cup \s{(s,r)} \\
  && {}\cup \s{(t_{j,k},w_k)} \\
  && {}\cup \s{(s,w_k)} \\
  && {}\cup \s{(v_i,t_{j,k}) \mid c_i\in A_j}. \\
\end{array}
\]
The graph is shown schematically in Figure~\ref{fig:graph}. In words,
it can be described as follows: for each cell $c_i$ of the
avoider-enforcer game, we have a pair of a black vertex $v_i$ and a
white vertex $u_i$, connected to each other by an edge. We also have a
cluster of $K$ black vertices $w_1,\ldots,w_K$. For each avoider set
$A_j$, we have a cluster of $K$ white vertices
$t_{j,1}\ldots,t_{j,K}$, which are connected to those $v_i$ such that
$c_i\in A_j$. All of the vertices $w_k$ and $t_{j,k}$ form a complete
bipartite graph. Finally, there is a \emph{universal} black vertex $r$
which is connected to all white vertices, and a \emph{universal} white
vertex $s$ which is connected to all black vertices (including $r$).

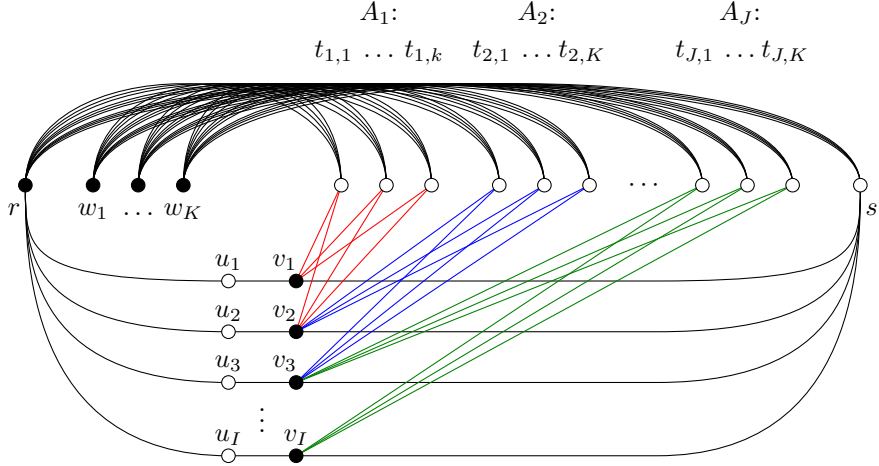
\begin{figure}
  \[
  \begin{tikzpicture}[scale=0.3]
    \def\xa{2} 
    \def\xb{3} 
    \def\woffset{-9}
    \def\toffset{5}
    \def\xc{2}
    \def\rcoord{\woffset}
    \def\wcoord#1{\woffset+\xb+\xa*#1}
    \def\tcoord#1#2#3{\toffset+\xa*#2+#1*\xb+#1*2*\xa+#3*\xc}
    \def\scoord{\toffset+3*\xb+3*2*\xa+\xc}
    \def\labelsep{4mm}
    \def\tlabelsep{17mm}
    \def\alabelsep{22mm}
    \def\ulabelsep{2mm}
    \def\yoffset{-2}
    \def\ya{2.25}
    \def\yc{-1}
    \def\ucoord#1#2{\yoffset-#1*\ya+#2*\yc}
    \def\voffset{3}
    \def\vplusoffset{19}
    \def\bend{6}
    \def\tshift{-1mm}
    
    \foreach\k in {0,...,2} {
      \foreach\j in {0,...,1} {
        \foreach\l in {0,...,2} {
          \draw (\wcoord{\k},0) .. controls +(0,\bend) and +(0,\bend) .. (\tcoord{\j}{\l}{0},0);
        }
      }
    }
    \foreach\k in {0,...,2} {
      \foreach\l in {0,...,2} {
        \draw (\wcoord{\k},0) .. controls +(0,\bend) and +(0,\bend) .. (\tcoord{2}{\l}{1},0);
      }
    }
    \foreach\k in {0,...,2} {
      \draw (\wcoord{\k},0) .. controls +(0,\bend) and +(0,\bend) .. (\scoord,0);
    }
    \foreach\j in {0,...,1} {
      \foreach\l in {0,...,2} {
        \draw (\rcoord,0) .. controls +(0,\bend) and +(0,\bend) .. (\tcoord{\j}{\l}{0},0);
      }
    }
    \foreach\l in {0,...,2} {
      \draw (\rcoord,0) .. controls +(0,\bend) and +(0,\bend) .. (\tcoord{2}{\l}{1},0);
    }
    \draw (\rcoord,0) .. controls +(0,\bend) and +(0,\bend) .. (\scoord,0);

    \foreach\i in {1,...,3} {
      \draw (\rcoord,0) .. controls +(0,-0.5*\bend) and +(-1.5*\bend,0) .. (0,\ucoord{\i}{0});
      \draw ((0,\ucoord{\i}{0}) -- ((\voffset,\ucoord{\i}{0});
      \draw ((\voffset,\ucoord{\i}{0}) -- ((\vplusoffset,\ucoord{\i}{0}) .. controls +(1.5*\bend,0) and +(0,-0.5*\bend) .. (\scoord,0);
    }
    \draw (\rcoord,0) .. controls +(0,-0.5*\bend) and +(-1.5*\bend,0) .. (0,\ucoord{4}{1});
    \draw ((0,\ucoord{4}{1}) -- ((\voffset,\ucoord{4}{1});
    \draw ((\voffset,\ucoord{4}{1}) -- (\vplusoffset,\ucoord{4}{1}) .. controls +(1.5*\bend,0) and +(0,-0.5*\bend) .. (\scoord,0);

    \foreach\i in {1,2} {
      \foreach\k in {0,...,2} {
        \draw[red] (\voffset,\ucoord{\i}{0}) -- (\tcoord{0}{\k}{0},0);
      }
    }
    
    \foreach\i in {2,3} {
      \foreach\k in {0,...,2} {
        \draw[blue] (\voffset,\ucoord{\i}{0}) -- (\tcoord{1}{\k}{0},0);
      }
    }

    \foreach\i in {3} {
      \foreach\k in {0,...,2} {
        \draw[green!50!black] (\voffset,\ucoord{\i}{0}) -- (\tcoord{2}{\k}{1},0);
      }
    }

    \foreach\i in {4} {
      \foreach\k in {0,...,2} {
        \draw[green!50!black] (\voffset,\ucoord{\i}{1}) -- (\tcoord{2}{\k}{1},0);
      }
    }
    
    \path (\rcoord,0) node[xshift=-1.5mm,below=\labelsep,anchor=base] {$r$};
    \path (\wcoord{0},0) node[below=\labelsep,anchor=base] {$w_1$};
    \path (\wcoord{1},0) node[below=\labelsep,anchor=base] {$\ldots$};
    \path (\wcoord{2},0) node[below=\labelsep,anchor=base] {$w_K$};
    \path (\tcoord{0}{0}{0},0) node[xshift=\tshift,above=\tlabelsep,anchor=base] {$t_{1,1}$};
    \path (\tcoord{0}{1}{0},0) node[xshift=\tshift,above=\tlabelsep,anchor=base] {$\ldots$};
    \path (\tcoord{0}{2}{0},0) node[xshift=\tshift,above=\tlabelsep,anchor=base] {$t_{1,k}$};
    \path (\tcoord{1}{0}{0},0) node[xshift=\tshift,above=\tlabelsep,anchor=base] {$t_{2,1}$};
    \path (\tcoord{1}{1}{0},0) node[xshift=\tshift,above=\tlabelsep,anchor=base] {$\ldots$};
    \path (\tcoord{1}{2}{0},0) node[xshift=\tshift,above=\tlabelsep,anchor=base] {$t_{2,K}$};
    \path (\tcoord{2}{0}{1},0) node[xshift=\tshift,above=\tlabelsep,anchor=base] {$t_{J,1}$};
    \path (\tcoord{2}{1}{1},0) node[xshift=\tshift,above=\tlabelsep,anchor=base] {$\ldots$};
    \path (\tcoord{2}{2}{1},0) node[xshift=\tshift,above=\tlabelsep,anchor=base] {$t_{J,K}$};
    \path (\tcoord{1.5}{1}{0.5},0) node {$\cdots$};
    \path (\scoord,0) node[xshift=1.5mm,below=\labelsep,anchor=base] {$s$};

    \path (\tcoord{0}{1}{0},0) node[xshift=\tshift,above=\alabelsep,anchor=base] {$A_1$:};
    \path (\tcoord{1}{1}{0},0) node[xshift=\tshift,above=\alabelsep,anchor=base] {$A_2$:};
    \path (\tcoord{2}{1}{1},0) node[xshift=\tshift,above=\alabelsep,anchor=base] {$A_J$:};

    \blackvertex(\rcoord,0);
    \foreach\k in {0,...,2} {
      \blackvertex(\wcoord{\k},0);
    }
    \foreach\j in {0,...,1} {
      \foreach\k in {0,...,2} {
        \whitevertex(\tcoord{\j}{\k}{0},0);
      }
    }
    \foreach\k in {0,...,2} {
      \whitevertex(\tcoord{2}{\k}{1},0);
    }
    \whitevertex(\scoord,0);

    \foreach\i in {1,...,3} {
      \whitevertex(0,\ucoord{\i}{0});
      \blackvertex(\voffset,\ucoord{\i}{0});
      \path (0,\ucoord{\i}{0}) node[above=\ulabelsep,anchor=base] {$u_{\i}$};
      \path (\voffset,\ucoord{\i}{0}) node[xshift=-1.5mm,above=\ulabelsep,anchor=base] {$v_{\i}$};
    }
    \whitevertex(0,\ucoord{4}{1});
    \blackvertex(\voffset,\ucoord{4}{1});
      \path (0,\ucoord{4}{1}) node[above=\ulabelsep,anchor=base] {$u_{I}$};
      \path (\voffset,\ucoord{4}{1}) node[above=\ulabelsep,anchor=base] {$v_{I}$};

    \path (0.5*\voffset,\ucoord{3.5}{0.5}) node[above=-3mm] {$\vdots$};
    
  \end{tikzpicture}
  \]
  \caption{The graph $\G{K}{(C,\Aa)}$. Here, we have assumed that
    $A_1=\s{c_1,c_2}$, $A_2=\s{c_2,c_3}$, and $A_J=\s{c_3,c_I}$.}
  \label{fig:graph}
\end{figure}

\subsection{Intended play, shenanigans, and simulation}

Our eventual goal is to show that whoever wins the game $(C,\Aa)$ also
wins the corresponding game $\G{K}{(C,\Aa)}$. To do so, we consider
the following notion of ``intended play'' for the Paintbucket
game. Note that there are $I$ pairs of vertices $\s{u_i,v_i}$. The
intended play is that the first $I$ moves of the game occur on these
vertices, with each pair being played in exactly once. The players are
not obliged to play in the intended way, but we will show that any
player who deviates from the intended play will lose the game
immediately. This is proved in the next two lemmas. When a player
deviates from intended play, we call it a \emph{shenanigan} by that
player.

\begin{lemma}[White shenanigans]\label{lem:shenanigan-white}
  Assume $K\geq|C|+2$, and suppose it is White's turn in the
  Paintbucket game on the graph $\G{K}{(C,\Aa)}$. If White's next move
  is not at a vertex of type $v$, White loses.
\end{lemma}

\begin{proof}
  There are two cases. Case 1: White's move is at the universal vertex
  $r$, then all white vertices will be contracted into a single
  vertex, while at least $w_1,\ldots,w_K$ remain as black
  vertices. Then Black wins on the next move by playing the remaining
  white vertex. Case 2: White's move is at a vertex of type $w$, say
  at $w_k$. Since all of the vertices of types $t$ and $s$ are
  neighbors of $w_k$, they all get contracted into a single white
  vertex, which we may still call $s$. This turns all of the $w_{k'}$
  where $k'\neq k$ into leaves by Lemma~\ref{lem:neighbors}. Hence,
  the resulting graph has at least $K-1$ black leaves attached to the
  single vertex $s$. Also note that there are at most $I+1$ white
  vertices remaining in the graph, namely $u_1,\ldots,u_I$ and $s$. By
  assumption, $K\geq |C|+2$, so $I+1\leq K-1$. Then Black wins by
  Lemma~\ref{lem:claw}.
\end{proof}

\begin{lemma}[Black shenanigans]\label{lem:shenanigan-black}
  Assume $K\geq|C|+2$ and $\Aa\neq\emptyset$.  Suppose it is Black's
  turn in the Paintbucket game on the graph $\G{K}{(C,\Aa)}$. If Black's
  next move is not at a vertex of type $u$, Black loses.
\end{lemma}

\begin{proof}
  By assumption, $\Aa$ is non-empty, so let $j$ be an index of some
  avoider set $A_j\in\Aa$.  There are two cases. Case 1: Black's move
  is at the universal vertex $s$, then all black vertices will be
  contracted into a single vertex, while at least $K$ white vertices
  remain, namely $t_{j,1},\ldots,t_{j,K}$. Then White wins on the next
  move by playing at the remaining black vertex.  This leaves the case
  where Black's move is at a vertex of type $t$, say at
  $t_{j,k}$. Since all of $w_1,\dots,w_K$ and $r$ are neighbors of
  $t_{j,k}$, they will all be contracted into a single vertex, which
  we may still call $r$. Because $t_{j,1},\ldots,t_{j,K}$ all had the
  same neighbors before Black's move, by Lemma~\ref{lem:neighbors},
  the remaining $K-1$ of them turn into leaves which share $r$ as a
  common neighbor.  Moreover, the resulting graph has at most $I+1$
  black vertices remaining, namely $v_1,\ldots,v_I$ and $r$. By
  assumption, $K\geq |C|+2$, so $I+1\leq K-1$. Therefore, White wins
  by the dual of Lemma~\ref{lem:claw}.
\end{proof}

\begin{lemma}[Simulation]\label{lem:simulation}
  Consider a position $(C,\Aa)$ of the avoider-enforcer game and let
  $K$ be a natural number. In the game of Paintbucket on the graph
  $\G{K}{(C,\Aa)}$, if Black moves at $u_i$, the resulting graph is
  isomorphic to $\G{K}{(C',\Aa')}$, where
  \[
  C' = C \setminus \s{c}
  \quad\mbox{and}\quad
  \Aa' = (A_j\setminus\s{c})_{j\in J}
  \]
  Similarly, if White moves at $v_i$, the resulting graph is isomorphic
  to $\G{K}{(C',\Aa')}$, where
  \[
  C'' = C \setminus \s{c}
  \quad\mbox{and}\quad
  \Aa'' = (A_j)_{j\in J,\,c\not\in A_j}.
  \]
  In other words, following Remark~\ref{rem:alternate}, these moves of
  Black and White exactly mirror the corresponding moves at $c_i$ in
  $(C,\Aa)$ by the avoider and enforcer, respectively.
\end{lemma}

\begin{proof}
  First, consider the Black move at $u_i$. The vertex $u_i$ has
  exactly two neighbors, namely $r$ and $v_i$. By the definition of
  Paintbucket on bipartite graphs, the result of the move is to delete
  $u_i$ and merge $v_i$ into $r$. Note that, since $r$ is a universal
  vertex, all of the neighbors of $v_i$ are already neighbors of $r$,
  so $r$ gains no new neighbors. The move is effectively equivalent to
  deleting the pair of vertices $u_i,v_i$ from the graph, along with
  all of their incident edges. This is exactly equivalent to removing
  the cell $c$ from $C$ and all avoider sets.

  Second, consider the White move at $v_i$. The vertex $v_i$ has
  several neighbors, namely $u_i$, $s$, and for each avoider set $A_j$
  that contains $c_i$, all of the $t_{j,k}$. The effect of moving
  at $v_i$ is to delete $v_i$ and merge all of its neighbors into a
  single vertex, which we may still call $s$. Since $s$ is universal,
  it already has all black vertices as neighbors, and therefore gains
  no new neighbors. Thus, White's move is equivalent to simply
  deleting $v_i$, $u_i$, and all the clusters $\s{t_{j,k}}$
  corresponding to avoider sets $A_j$ that contain $c_i$. This is
  exactly equivalent to removing the cell $c$ from $C$ and removing all
  avoider sets containing $c$ from $\Aa$.
\end{proof}

\subsection{Proof of the main result}

The PSPACE-completeness of Paintbucket is a consequence of the
following proposition, which relates the winner of $\G{K}{(C,\Aa)}$ to
that of $(C,\Aa)$.

\begin{proposition}\label{pro:main}
  Consider a position $(C,\Aa)$ of the avoider-enforcer game, and
  assume that $K\geq |C|+2$.
  \begin{enumerate}\alphalabels
  \item If $|C|$ is even, Black has a first-player winning strategy in
    the Paintbucket game $\G{K}{(C,\Aa)}$ if and only if the
    avoider has a first-player winning strategy in $(C,\Aa)$.
  \item If $|C|$ is odd, Black has a second-player winning strategy in
    the Paintbucket game $\G{K}{(C,\Aa)}$ if and only if the
    avoider has a second-player winning strategy in $(C,\Aa)$.
  \end{enumerate}
\end{proposition}

\begin{proof}
  We prove properties (a) and (b) by simultaneous induction on $|C|$.
  We consider three cases. See also Figure~\ref{fig:simulation}.

\begin{figure}
  \[
  \begin{tikzpicture}[on grid, node distance=2cm and 4cm]
    \node (aegame) {Avoider-enforcer game:};
    \node (paintbucket) [right=of aegame] {Paintbucket:};
    \node (ca) [blue,yshift=1cm,below=of aegame] {$(C,\Aa)$};
    \node (gca) [blue,yshift=1cm,below=of paintbucket] {$\G{K}{(C,\Aa)}$};
    \node (ca') [blue,below=of ca] {$(C',\Aa')$};
    \node (gca') [blue,below=of gca] {$\G{K}{(C',\Aa')}$};
    \node (ca'') [blue,below=of ca'] {$(C'',\Aa'')$};
    \node (gca'') [blue,below=of gca'] {$\G{K}{(C'',\Aa'')}$};
    \node [left=of ca] {$|C|$ is even:};
    \node [left=of ca'] {$|C'|$ is odd:};
    \node (c''iseven) [left=of ca''] {$|C''|$ is even:};
    \node (blackloses) [red,xshift=-15mm,yshift=5mm,right=of gca'] {Black loses};
    \node (whiteloses) [red,xshift=-15mm,yshift=5mm,right=of gca''] {White loses};
    \node (avoiderloses) [blue,yshift=-8mm,xshift=-1cm,below=of ca''] {\begin{tabular}{c}Avoider\\loses\end{tabular}};
    \node (avoiderwins) [blue,yshift=-8mm,xshift=1cm,below=of ca''] {\begin{tabular}{c}Avoider\\wins\end{tabular}};
    \node (blackloses2) [blue,yshift=-8mm,xshift=-1cm,below=of gca''] {\begin{tabular}{c}Black\\loses\end{tabular}};
    \node (blackwins2) [blue,yshift=-8mm,xshift=1cm,below=of gca''] {\begin{tabular}{c}Black\\wins\end{tabular}};
    \node [below=of c''iseven,yshift=-8mm] {Base case:};
    \node (cadot) [yshift=1cm,below=of ca'',circle] {};
    \node (gcadot) [yshift=1cm,below=of gca'',circle] {};
    \draw[->,blue] (ca) -- node [left] {\small\sl\begin{tabular}{r}Avoider moves\\ at $c_i$\end{tabular}} (ca');
    \draw[->,blue] (gca) -- node [left] {\small\sl\begin{tabular}{r}Black moves\\ at $u_i$\end{tabular}} (gca');
    \draw[->,red] (gca) -- node [right,xshift=-1mm] {\small\sl\begin{tabular}{@{}l}Black moves\\ ~~~~~~elsewhere\end{tabular}} (blackloses);
    \draw[->,blue] (ca') -- node [left] {\small\sl\begin{tabular}{r}Enforcer moves\\ at $c_{i'}$\end{tabular}} (ca'');
    \draw[->,blue] (gca') -- node [left] {\small\sl\begin{tabular}{r}White moves\\ at $v_{i'}$\end{tabular}} (gca'');
    \draw[->,red] (gca') -- node [right,xshift=-1mm] {\small\sl\begin{tabular}{@{}l}White moves\\ ~~~~~~elsewhere\end{tabular}} (whiteloses);
    \path (ca) -- node {$\sim$} (gca);
    \path (ca') -- node {$\sim$} (gca');
    \path (ca'') -- node {$\sim$} (gca'');
    \path (ca'') -- node {$\vdots$} (cadot);
    \path (gca'') -- node {$\vdots$} (gcadot);
    \draw[->,blue] (cadot) -- node [left] {$\Aa=\emptyset$} (avoiderloses);
    \draw[->,blue] (cadot) -- node [right] {\begin{tabular}{@{}l}$\Aa\neq\emptyset$,\\~~$C=\emptyset$\end{tabular}} (avoiderwins);
    \draw[->,blue] (gcadot) -- node [left] {$\Aa=\emptyset$} (blackloses2);
    \draw[->,blue] (gcadot) -- node [right] {\begin{tabular}{@{}l}$\Aa\neq\emptyset$,\\~~$C=\emptyset$\end{tabular}}(blackwins2);
  \end{tikzpicture}
  \]  
  \caption{Schematic representation of the proof of Proposition~\ref{pro:main}}
  \label{fig:simulation}
\end{figure}

  Case 1. $\Aa=\emptyset$. Since there are no avoider sets,
  the avoider wins the avoider-enforcer game no matter how the players
  play. We must show that Black has a winning strategy in the
  Paintbucket game on $\G{K}{(C,\Aa)}$. Since $\Aa=\emptyset$, there
  are no vertices of type $t$. Therefore, $w_1,\ldots,w_K$ have $s$ as
  their only neighbor, so there are at least $K$ black leaves
  attached to a single vertex. Also, the graph has at most $I+1$ white
  vertices, namely $u_1,\ldots,u_I$ and $s$. We assumed $K\geq
  |C|+2=I+2$, therefore $I+1<K$. Then Black wins by
  Lemma~\ref{lem:claw}, no matter whose turn it is.

  Case 2. $\Aa\neq\emptyset$ and $C=\emptyset$. Then (b) is vacuously
  true. To prove (a), first note that the avoider-enforcer game is
  already finished and has been won by the enforcer (since there is an
  avoider set, which is necessarily empty, and of which the avoider
  has therefore claimed all members). So we must show that Black loses
  the game $\G{K}{(C,\Aa)}$, moving first. But since $C=\emptyset$,
  there are no vertices of types $u$. Black loses by
  Lemma~\ref{lem:shenanigan-black}.

  Case 3. $|C|>0$ and $|\Aa|>0$. To prove (a), assume $|C|$ is
  even. By Lemma~\ref{lem:shenanigan-black}, Black has no possible
  winning moves in the Paintbucket game except at vertices of type
  $u$. We claim that $u_i$ is a winning move for Black in the
  Paintbucket game $\G{K}{(C,\Aa)}$ if and only if $c_i$ is a
  winning move for the avoider in the avoider-enforcer game $(C,\Aa)$.
  Indeed, this follows from the simulation lemma and the induction
  hypothesis. Namely, if the avoider's move at $c_i$ results in the
  position $(C',\Aa')$, then Black's move at $u_i$ results in a
  position that is isomorphic to $\G{K}{(C',\Aa')}$. The move at
  $u_i$ is winning if and only if Black has a second-player winning
  strategy in $\G{K}{(C',\Aa')}$, which, by the induction hypothesis
  and the fact that $|C'|$ is odd, is the case if and only if the
  avoider has a second-player winning strategy in $(C',\Aa')$, which
  is the case if and only if the move at $c_i$ was winning.
  The proof of (b) is analogous but dual.  
\end{proof}

We can now prove the main theorem, i.e., the PSPACE-completeness of
Paintbucket on graphs.

\begin{proof}[Proof of Theorem~\ref{thm:main}]
  It is easy to see that Paintbucket is in PSPACE. To show
  PSPACE-hardness, consider any position $(C,\Aa)$ of the
  avoider-enforcer game. As noted in Remark~\ref{rem:even}, we can
  assume without loss of generality that $|C|$ is even and that the
  player to move is the avoider. We construct a position of
  Paintbucket as follows. Let $K=|C|+2$ and consider Paintbucket on
  the graph $\G{K}{(C,\Aa)}$. Note that the size of this graph is
  polynomial in the size of the avoider-enforcer instance. Then by
  Proposition~\ref{pro:main}, any algorithm that can determine the
  winner of the game $\G{K}{(C,\Aa)}$ can also determine the winner of
  $(C,\Aa)$; it follows that the decision problem for Paintbucket is
  at least as hard as the decision problem for the avoider-enforcer
  game. Therefore, it is PSPACE-hard.
\end{proof}

\section{Conclusion}

We showed that the game of Paintbucket on graphs is
PSPACE-complete. The obvious open question is whether the original
version of Paintbucket, played on a square grid, is also
PSPACE-complete. Our current method does not shed any light on this
question, but it would be fun to figure it out in future work.

\bibliographystyle{abbrv}
\bibliography{paintbucket}

\end{document}